\documentclass[oneside,english]{amsart}
\usepackage[T1]{fontenc}
\usepackage[latin9]{inputenc}
\usepackage{textcomp}
\usepackage{amsbsy}
\usepackage{amstext}
\usepackage{amsthm}
\usepackage{amssymb}

\makeatletter
\numberwithin{equation}{section}
\numberwithin{figure}{section}
\theoremstyle{plain}
\newtheorem*{thm*}{\protect\theoremname}
\theoremstyle{plain}
\newtheorem{thm}{\protect\theoremname}
\theoremstyle{definition}
\newtheorem{defn}[thm]{\protect\definitionname}
\theoremstyle{definition}
\newtheorem{example}[thm]{\protect\examplename}
\theoremstyle{remark}
\newtheorem{rem}[thm]{\protect\remarkname}
\theoremstyle{plain}
\newtheorem{prop}[thm]{\protect\propositionname}
\theoremstyle{plain}
\newtheorem{lem}[thm]{\protect\lemmaname}
\theoremstyle{plain}
\newtheorem{cor}[thm]{\protect\corollaryname}

\usepackage[a4paper]{geometry}

\theoremstyle{definition}

\theoremstyle{theorem}

\theoremstyle{problem}
\newtheorem*{question}{Question}

\usepackage{graphicx}
\usepackage{amsmath}
\usepackage{tikz}
\usepackage{tikz-3dplot}
\usetikzlibrary{math,calc,arrows}

\makeatother

\usepackage{babel}
\providecommand{\corollaryname}{Corollary}
\providecommand{\definitionname}{Definition}
\providecommand{\examplename}{Example}
\providecommand{\lemmaname}{Lemma}
\providecommand{\propositionname}{Proposition}
\providecommand{\remarkname}{Remark}
\providecommand{\theoremname}{Theorem}

\begin{document}
\title{Isometric Embeddings of Conformally Compact Manifolds into Hyperbolic
Spaces}
\author{Marco Usula}
\begin{abstract}
The celebrated Nash Embedding Theorem asserts that every closed Riemannian
manifold can be isometrically embedded into a sufficiently high-dimensional
Euclidean space. In this paper, we prove an analogous result in the
conformally compact context. Let $\left(M,g\right)$ be a conformally
compact manifold whose sectional curvature at infinity is strictly
bounded below by a negative constant $-\lambda^{2}$. We prove that
$\left(M,g\right)$ can be realized as a submanifold, transverse to
the sphere at infinity, of a sufficiently high-dimensional rescaled
hyperbolic space of constant curvature $-\lambda^{2}$.
\end{abstract}

\maketitle

\section{Introduction}

The Nash Embedding Theorem is a result of fundamental philosophical
importance in differential geometry. The theorem states that every
smooth Riemannian manifold $\left(M,g\right)$ can be embedded smoothly
and isometrically in Euclidean space $\mathbb{R}^{N}$, for $N$ sufficiently
large. This theorem was proved by John Nash in 1954---1956; the first
version of the theorem asserts the existence of $C^{1}$ isometric
embeddings \cite{NashC1IsometricImbeddings}, while the second version
generalizes this result to $C^{\infty}$ embeddings \cite{NashImbeddingProblem}.
Besides the intrinsic interest of the result, the paper \cite{NashImbeddingProblem}
introduced an extremely important tool in modern mathematics, the
\emph{Nash--Moser Technique}, which is central in many existence
theorems for solutions of PDEs. The original proof of Nash was subsequently
simplified drastically by Günther \cite{GuntherIsometricEmbeddings, GuntherPerturbationProblem},
who managed to reduce the proof to a standard fixed point argument,
by a clever use of elliptic theory. We refer to \cite{YangNashIsometricEmbedding, TaoNashEmbedding}
for expositions on this alternative proof. The literature on the Nash
Embedding Theorem and the Nash--Moser Technique is vast, and testifies
to the large impact of these results on the mathematical community.
We point the reader to \cite{GromovPDR, GromovDescendants} for a
much more in-depth discussion on the Nash Embedding Theorem, and the
current trends of research stemming from it.

The aim of this paper is to prove a similar embedding theorem, in
the context of \emph{conformally compact geometry}. A conformally
compact manifold is a compact manifold with boundary $M$, equipped
with a metric $g$ defined on the \emph{interior} of $M$, such that
if $x$ is a boundary defining function\footnote{A smooth function $x:M\to[0,+\infty)$ such that $x^{-1}\left(0\right)=\partial M$
and $0$ is a regular value for $x$.} for $M$, then the conformal rescaling $x^{2}g$ extends smoothly
to a metric on the whole of $M$. The typical example of a conformally
compact manifold is hyperbolic space: the hyperbolic metric can be
seen as the conformally compact metric
\[
\frac{4d\boldsymbol{y}^{2}}{\left(1-\left|\boldsymbol{y}\right|^{2}\right)^{2}}
\]
on the interior of the unit ball in $\mathbb{R}^{n}$. Analogously
to the hyperbolic metric, conformally compact metrics are complete
with bounded geometry, they induce a conformal class on the boundary
(the so-called \emph{conformal infinity} of the metric), and they
are asymptotically negatively curved at infinity. More precisely,
each conformally compact manifold $\left(M,g\right)$ has a natural
Riemannian invariant called the \emph{sectional curvature at infinity}:
this is a strictly negative function $\kappa_{\infty}^{g}\in C^{\infty}\left(\partial M\right)$,
such that for every sequence of tangent $2$-planes to $M^{\circ}$
converging to a tangent $2$-plane to $M$ at a point $p\in\partial M$,
the corresponding sequence of sectional curvatures converges to $\kappa_{\infty}^{g}\left(p\right)$.
When $\kappa_{\infty}^{g}\equiv-1$ identically, these metrics are
called \emph{asymptotically hyperbolic}. Conformally compact and asymptotically
hyperbolic manifolds have been intensely studied since the 80s, starting
with the pioneering works of Fefferman--Graham \cite{FeffermanGraham},
Mazzeo \cite{MazzeoPhD, MazzeoHodge} and Mazzeo--Melrose \cite{MazzeoMelroseResolvent}.
Of particular interest are the conformally compact and \emph{Einstein
}manifolds, also called \emph{Poincaré--Einstein }manifolds, which
play a central role in the Riemannian version of the AdS-CFT correspondence
\cite{Maldacena, WittenHolography, Biquard_Metriques, GrahamLee, LeeFredholm, UsulaYM, LimaEinsteinYangMills}.

Denote by $\mathrm{H}^{N+1}\left(-\lambda^{2}\right)$ the rescaled
hyperbolic space of constant sectional curvature $-\lambda^{2}$.
We consider compact submanifolds with boundary $M\subset\mathrm{H}^{N+1}\left(-\lambda^{2}\right)$,
such that $\partial M$ is contained in the sphere at infinity of
$\mathrm{H}^{N+1}\left(-\lambda^{2}\right)$, $M^{\circ}$ is contained
in the interior of $\mathrm{H}^{N+1}\left(-\lambda^{2}\right)$, and
moreover $M$ is transverse to the boundary at infinity. Following
Melrose's terminology developed in the book \cite{MelroseCorners},
we call these submanifolds \emph{interior p-submanifolds}. This class
of submanifolds is natural in conformally compact geometry; indeed,
we prove in Proposition \ref{prop:p-submanifolds-of-CC-are-CC} that
if $M$ is an interior p-submanifold of a conformally compact manifold
$\left(N,h\right)$, then the ambient conformally compact metric $h$
on $N$ induces a conformally compact metric $g$ on $M$. In this
case, the sectional curvature at infinity of $g$ is pointwise bounded
below by the sectional curvature at infinity of $h$ along $\partial M$.

We prove in Proposition \ref{prop:p-submanifolds-of-CC-are-CC} that
every interior p-submanifold $\left(M,g\right)$ of $\mathrm{H}^{N+1}\left(-\lambda^{2}\right)$
must satisfy the pointwise inequality $\kappa_{\infty}^{g}\geq-\lambda^{2}$.
The natural question that arises is: if $\left(M,g\right)$ is a conformally
compact manifold with $\kappa_{\infty}^{g}\geq-\lambda^{2}$, can
$\left(M,g\right)$ be realized as an interior p-submanifold of $\mathrm{H}^{N+1}\left(-\lambda^{2}\right)$,
for $N$ sufficiently large? In this paper, we answer this question
positively if the inequality is \emph{strict}. More precisely, we
prove the following
\begin{thm*}
Let $\left(M^{m},g\right)$ be a $m$-dimensional conformally compact
manifold, whose sectional curvature at infinity satisfies the pointwise
inequality $\kappa_{\infty}^{g}>-\lambda^{2}$ for some $\lambda\not=0$.
Then $\left(M,g\right)$ can be realized as an interior p-submanifold
of $\mathrm{H}^{\mathrm{N}_{m}+1}\left(-\lambda^{2}\right)$, where
$\mathrm{N}_{m}$ is the smallest positive integer such that every
closed $m$-dimensional Riemannian manifold embeds isometrically into
$\mathbb{R}^{\mathrm{N}_{m}}$.
\end{thm*}
The key part of the proof consists in constructing a boundary defining
function $x$ for $M$ such that the symmetric $2$-tensor $x^{2}g-\lambda^{-2}dx^{2}$
is positive-definite, and then applying the standard Nash Embedding
Theorem.

The proof proposed in this paper breaks if the equality in $\kappa_{\infty}^{g}\geq-\lambda^{2}$
is achieved at some point of the boundary. In particular, the paper
leaves open the following question: if $\left(M,g\right)$ is asymptotically
hyperbolic, can it be isometrically embedded into $\mathrm{H}^{N+1}\left(-1\right)$
for $N$ sufficiently large? The impression of the author is that
the answer to this question may be positive, but that its proof (if
the statement is true) may require substantial work beyond the techniques
adopted in this paper.

We conclude by discussing potential applications of this result. Besides
its intrinsic interest, the Nash Embedding Theorem has been used to
facilitate many arguments in geometric analysis. For example, in order
to define Sobolev spaces of maps between closed Riemannian manifolds,
it is convenient to embed the target manifold isometrically into some
Euclidean space \cite{HajlaszSobolevMappings}; this is particularly
useful when studying variational problems for maps. The author of
this paper has been recently interested in harmonic and biharmonic
maps between conformally compact manifolds \cite{UsulaBiharmonic};
if formulated in the appropriate way, these are both nonlinear ``$0$-elliptic''
equations, and as such, its solutions are expected to be polyhomogeneous\footnote{A function on a compact manifold with boundary is polyhomogeneous
if it is smooth in the interior and it admits a Taylor-like asymptotic
expansion at the boundary, with terms of the form $x^{\alpha}\left(\log x\right)^{l}$
for a boundary defining function $x$, $\alpha\in\mathbb{C}$, and
$l\in\mathbb{N}$.} provided that the boundary data is smooth, and the solutions satisfy
some weak interior regularity condition sufficient for a boot-strap
argument. The author was motivated to prove the main result of this
paper precisely in order to facilitate the proofs of these regularity
results. This topic will be developed in detail in an upcoming paper.
Besides this specific application, the author expects that the embedding
theorem presented here can find many applications in conformally compact
geometry, inspired by analogous applications of the Nash Embedding
Theorem as described, for example, in \cite{GromovPDR, GromovDescendants}.

\subsection*{Acknowledgements}

The author wishes to thank Rafe Mazzeo and Andrea Loi, for their encouragement
during the completion of this work.

\section{Conformally Compact Geometry}

Let $M$ be a compact manifold with boundary. Denote by $M^{\circ}$
the interior of $M$, and by $\partial M$ the boundary of $M$. In
this section, we recall the basic notions of conformally compact geometry.
Recall that a \emph{boundary defining function }for $M$ is a smooth
function $x:M\to[0,+\infty)$ such that $x^{-1}\left(0\right)=\partial M$
and $dx$ is nowhere vanishing along $\partial M$.
\begin{defn}
A \emph{conformally compact metric} on $M$ is a metric $g$ on the
interior $M^{\circ}$ with the following properties: for some (hence
every) boundary defining function $x$ on $M$, the rescaling $x^{2}g$
extends smoothly to a metric on $M$.
\end{defn}

\begin{example}
The paradigmatic example of a conformally compact manifold is hyperbolic
space. Denote by $B^{n}$ the closed unit ball in $\mathbb{R}^{n}$,
with coordinates $\boldsymbol{y}=\left(y^{1},...,y^{n}\right)$. The
hyperbolic metric is the metric on the interior of $B^{n}$ defined
as
\[
h=\frac{4d\boldsymbol{y}^{2}}{\left(1-\left|\boldsymbol{y}\right|^{2}\right)^{2}}.
\]
Call $\rho=1-\left|\boldsymbol{y}\right|^{2}$. It is immediate to
check that $\rho$ is a boundary defining function for $B^{n}$, and
$\rho^{2}h$ extends smoothly to the constant multiple $4d\boldsymbol{y}^{2}$
of the Euclidean metric on $B^{n}$.
\end{example}

Conformally compact manifolds possess many properties in common with
hyperbolic space: they are complete, with bounded geometry, and they
have infinite volume. Moreover, the boundary $\partial M$ is at infinite
distance from any point in the interior. A conformally compact metric
$g$ on $M$ does not induce a well-defined metric on $\partial M$;
rather, it induces a \emph{conformal class}, defined by
\[
\mathfrak{c}_{\infty}\left(g\right):=\left\{ \left(x^{2}g\right)_{|\partial M}:\text{\ensuremath{x} is a boundary defining function}\right\} .
\]
This conformal class is usually called the \emph{conformal infinity}
of $\left(M,g\right)$. In the case of $n+1$-dimensional hyperbolic
space, the conformal infinity is the round conformal class on the
$n$-sphere.

Strictly speaking, a conformally compact metric $g$ on $M$ is only
a metric in the interior. However, $g$ should really be thought of
as a smooth object on the whole of \emph{ }$M$. There is a very elegant
way to make this statement precise. This idea is formulated in §2.A
of \cite{MazzeoPhD}, inspired by an analogous construction of Melrose--Mendoza
\cite{MelroseMendozaTotallyCharacteristic, MelroseAPS}. Denote by
$\mathcal{V}_{0}\left(M\right)$ the space of smooth vector fields
on $M$ which vanish along the boundary; we call the elements of $\mathcal{V}_{0}\left(M\right)$
\emph{$0$-vector fields} on $M$. The space $\mathcal{V}_{0}\left(M\right)$
is a locally finitely generated, projective $C^{\infty}\left(M\right)$
module; therefore, by the Serre--Swan Theorem, $\mathcal{V}_{0}\left(M\right)$
can be realized as the module of sections of a smooth vector bundle
over $M$, called the \emph{$0$-tangent bundle}. We denote this bundle
by $^{0}TM$.

Every tensor bundle over $M$ is a bundle associated to the principal
$\text{GL}\left(m,\mathbb{R}\right)$ bundle $\text{Fr}\left(TM\right)\to M$
via a real linear representation of $\text{GL}\left(m,\mathbb{R}\right)$.
We can then define the corresponding ``$0$-tensor bundle'', by
replacing $\text{Fr}\left(TM\right)$ with the bundle $\text{Fr}\left(^{0}TM\right)$.
We mention in particular:
\begin{enumerate}
\item the \emph{$0$-cotangent bundle} $^{0}T^{*}M$, i.e. the dual of $^{0}TM$;
\item the \emph{bundle of $0$-$k$-forms }$^{0}\Lambda^{k}$, i.e. the
$k$-th exterior power of $^{0}T^{*}M$;
\item the\emph{ bundle of symmetric $0$-$2$-tensors}, $S^{2}\left(^{0}T^{*}M\right)$,
i.e. the second symmetric power of $^{0}T^{*}M$.
\end{enumerate}
If $E$ is a tensor bundle over $M$, and $^{0}E$ is the corresponding
$0$-tensor bundle, then $E$ and $^{0}E$ are canonically isomorphic
along the interior $M^{\circ}$. To see that, we define the \emph{anchor
map }$\#:{}^{0}TM\to TM$ as the bundle map induced by the $C^{\infty}\left(M\right)$
linear map on sections $\mathcal{V}_{0}\left(M\right)\to\mathcal{V}\left(M\right)$.
Since $0$-vector fields are unconstrained in the interior, the map
$\#_{p}:{^{0}T_{p}M}\to T_{p}M$ is an isomorphism at each $p\in M^{\circ}$.
This implies that we have canonical identifications $E_{|M^{\circ}}\equiv{^{0}E}_{|M^{\circ}}$.
These isomorphisms in general do not extend over the boundary: indeed,
the uniform degeneracy of $0$-vector fields along $\partial M$ implies
that the anchor map $\#$ \emph{vanishes identically }along $\partial M$.
As an effect of this, if $E$ is a tensor bundle of type $\left(r,s\right)$
(i.e. $r$ contravariant and $s$ covariant indices), then the sections
of the corresponding $0$-tensor bundle $^{0}E$ are precisely of
the form $\omega=x^{\left(r-s\right)}\overline{\omega}$, where $x$
is a boundary defining function and $\overline{\omega}$ is a smooth
section of $E$. For example, smooth sections of $^{0}\Lambda^{k}$
are precisely ``uniformly singular $k$-forms'' which can be written
as 
\[
\omega=\frac{\overline{\omega}}{x^{k}}
\]
where $\overline{\omega}$ is a smooth $k$-form on $M$. It is now
immediate to see that conformally compact metrics on $M$ are precisely
smooth, fibrewise positive definite symmetric $0$-$2$-tensors, i.e.
the smooth bundle metrics on $^{0}TM$. For this reason, conformally
compact metrics are sometimes called \emph{$0$-metrics}.

A fundamental property of conformally compact metrics is that they
are \emph{asymptotically negatively curved}. Let us be more precise.
\begin{defn}
Let $\left(M,g\right)$ be a conformally compact manifold. The \emph{sectional
curvature at infinity} of $g$ is the function
\[
\kappa_{\infty}^{g}:=-\left|\frac{dx}{x}\right|_{g|\partial M}^{2},
\]
where $x$ is a boundary defining function on $M$.
\end{defn}

\begin{rem}
\label{rem:dx/x_canonical}Note that the definition above is well-posed,
i.e. it is independent of $x$. To see this we first note that, although
the $1$-form $dx/x$ is singular when seen as a section of $\Lambda^{1}$,
it is in fact smooth up to the boundary as a section of $^{0}\Lambda^{1}$;
since $g$ is a bundle metric on $^{0}TM$, it induces a dual bundle
metric on $^{0}\Lambda^{1}$, so $\left|dx/x\right|_{g}^{2}$ is well-defined
up to the boundary. Now, although $dx/x$ clearly depends on $x$,
\emph{its restriction to $\partial M$ does not}. Indeed, if $\tilde{x}$
is another boundary defining function on $M$, we can write $\tilde{x}=e^{\varphi}x$
for some $\varphi\in C^{\infty}\left(M\right)$, so that
\[
\frac{d\tilde{x}}{\tilde{x}}=\frac{dx}{x}+d\varphi.
\]
The differential $d\varphi$ is a smooth section of $\Lambda^{1}=x\ {^{0}\Lambda^{1}}$;
therefore, as a section of $^{0}\Lambda^{1}$, it vanishes identically
along $\partial M$.
\end{rem}

The name of the invariant $\kappa_{\infty}^{g}$ defined above is
justified by the following result, due to Mazzeo:
\begin{prop}
(Proposition 1.10 of \cite{MazzeoPhD}) Let $\left(M,g\right)$ be
conformally compact. Let $\left\{ p_{k}\right\} _{k\in\mathbb{N}}\subset M^{\circ}$
be a sequence of points converging to a boundary point $p_{\infty}\in\partial M$,
and let $\pi_{k}\subseteq T_{p_{k}}M^{\circ}$ be a sequence of tangent
$2$-planes. Denote by $\kappa^{g}\left(\pi_{k}\right)$ the sectional
curvature of $g$ at $\pi_{k}$. Then
\[
\lim_{k\to+\infty}\kappa^{g}\left(\pi_{k}\right)=\kappa_{\infty}^{g}\left(p_{\infty}\right).
\]
\end{prop}

The conformally compact metrics with curvature at infinity constantly
equal to $-1$ have a special name:
\begin{defn}
A conformally compact metric $g$ on $M$ is called \emph{asymptotically
hyperbolic} if $\kappa_{\infty}^{g}\equiv-1$.
\end{defn}

\section{Interior p-Submanifolds of Conformally Compact Manifolds}

Let $M$ be a compact manifold with boundary. By a (compact) \emph{submanifold}
of $M$, we mean the image of an injective immersion $\iota:S\to M$,
where $S$ is a compact manifold with or without boundary. Since $S$
is compact, the push-forward topology on $\iota\left(S\right)$ coincides
automatically with the subspace topology induced by $M$.
\begin{defn}
An \emph{interior p-submanifold} of $M$ is a compact submanifold
with boundary $S$ with the following properties:
\begin{enumerate}
\item $S^{\circ}\subset M^{\circ}$ and $\partial S\subset\partial M$
\item $S$ is transverse to $\partial M$.
\end{enumerate}
\end{defn}

The terminology ``interior p-submanifold'' is due to Melrose, cf.
§1.7 of \cite{MelroseCorners}. The letter ``p'' stands for ``product'',
because interior p-submanifolds admit tubular neighborhoods in $M$.
The notion of interior p-submanifold is related to another concept
due to Melrose:
\begin{defn}
Let $N$ be a compact manifold with boundary. A map\emph{ $u:M\to N$}
is said to be an \emph{interior $b$-map} if, given a boundary defining
function $X$ for $N$ and a boundary defining function $x$ for $M$,
the pull-back $u^{*}X$ can be written as $e^{\varphi}x^{k}$ for
some $\varphi\in C^{\infty}\left(M\right)$ and $k\in\mathbb{N}$.
\end{defn}

It is straightforward to check that the integer $k$ of the previous
definition does not depend on $X$ and $x$. Following \cite{UsulaBiharmonic},
when $k=1$ we say that $u:M\to N$ is a \emph{simple $b$-map}. In
other words, a map $u:M\to N$ is a simple $b$-map if and only if
the pull-back of a boundary defining function on $N$ is a boundary
defining function on $M$.
\begin{lem}
Let $u:M\to N$ be a smooth map such that $\partial M\mapsto\partial N$
and $M^{\circ}\mapsto N^{\circ}$. Then $u$ is a simple $b$-map
if and only if it is transverse to $\partial N$.
\end{lem}

\begin{proof}
Let $X$ be a boundary defining function for $N$, and let $x=u^{*}X$.
Since $X$ is a boundary defining function, and $\partial M\mapsto\partial N$
and $M^{\circ}\mapsto N^{\circ}$, it follows that $x^{-1}\left(0\right)=\partial M$.
It remains to prove that $dx$ is nowhere zero along $\partial M$.
For every $p\in\partial M$, we have
\[
dx_{p}=d\left(X\circ u\right)_{p}=dX_{u\left(p\right)}\circ du_{p}.
\]
Since $X$ is a boundary defining function for $N$, for every $q\in\partial N$
the kernel of $dX_{q}$ coincides with $T_{q}\partial N$. Therefore,
for every $p\in\partial M$, we have $dx_{p}=0$ if and only if $du_{p}$
maps $T_{p}M$ into $T_{u\left(p\right)}\partial N$, i.e. if and
only if $u$ is not transverse to $\partial N$ at $p$.
\end{proof}
This lemma immediately implies the following
\begin{cor}
\label{cor:p-submanifold-iff-image-immersion-bmap}A submanifold $S\subset M$
is an interior p-submanifold if and only if the inclusion $\iota:S\to M$
is an injective immersion and a simple $b$-map. In particular, if
$S\subset M$ is an interior p-submanifold and $X$ is a boundary
defining function for $M$, then the restriction $X_{|S}$ is a boundary
defining function for $S$.
\end{cor}

Fix now a conformally compact metric $g$ on $M$. We want to show
that, if $S$ is an interior p-submanifold of $M$, then $S$ inherits
from $g$ a conformally compact metric. The key to this result is
the following
\begin{lem}
Let $u:M\to N$ be a simple $b$-map between compact manifolds with
boundary. Denote by $u^{\circ}:M^{\circ}\to N^{\circ}$ the restriction
of $u$ to the interiors. Then the differential $du^{\circ}:TM^{\circ}\to TN^{\circ}$
extends from the interior to a smooth bundle map $^{0}du:{^{0}TM}\to{^{0}TN}$
covering $u$, which we call the \emph{$0$-differential} of $u$.
Moreover, if $u$ is an immersion, $^{0}du$ is injective.
\end{lem}

\begin{proof}
Let $p\in\partial M$, and call $q=u\left(p\right)$. Choose half-space
coordinates $\left(x,y^{1},...,y^{m}\right)$ for $M$ centered at
$p$, and half-space coordinates $\left(X,Y^{1},...,Y^{n}\right)$
for $N$ centered at $q$. Since $u$ is a simple $b$-map, we can
without loss of generality assume that $x=u^{*}X$. Thus, we can write
$u$ in these coordinates as $u=\left(x,u^{1},...,u^{n}\right)$.
In terms of the local frames $\partial_{x},\partial_{y^{i}}$ and
$\partial_{X},\partial_{Y^{j}}$ for $TM$ and $TN$, respectively,
we then have
\[
du=\left(\begin{matrix}1 & 0 & \cdots & 0\\
\frac{\partial u^{1}}{\partial x} & \frac{\partial u^{1}}{\partial y^{1}} & \cdots & \frac{\partial u^{1}}{\partial y^{m}}\\
\vdots & \vdots &  & \vdots\\
\frac{\partial u^{n}}{\partial x} & \frac{\partial u^{n}}{\partial y^{1}} & \cdots & \frac{\partial u^{n}}{\partial y^{m}}
\end{matrix}\right).
\]
Using the canonical identifications $^{0}TM_{|M^{\circ}}\equiv TM^{\circ}$
and $^{0}TN_{|N^{\circ}}\equiv TN^{\circ}$, we observe that this
expression is \emph{precisely} the expression of $du$ with respect
to the local frames $x\partial_{x},x\partial_{y^{i}}$ and $X\partial_{X},X\partial_{Y^{j}}$
for $^{0}TM$ and $^{0}TN$, in the interior. It follows that $^{0}du$
extends smoothly to a bundle map $^{0}TM\to{^{0}TN}$ as claimed,
and it is injective precisely when $du$ is injective.
\end{proof}


\begin{prop}
\label{prop:p-submanifolds-of-CC-are-CC}Let $\left(M,g\right)$ be
a conformally compact manifold, and let $S\subset M$ be an interior
p-submanifold. Call $\iota:S\to M$ the inclusion. Then $g_{|S}\equiv\iota^{*}g$
is a conformally compact metric on $S$. Moreover, if $\kappa_{\infty}^{g}$
and $\kappa_{\infty}^{\iota^{*}g}$ denote the sectional curvatures
at infinity of $\left(M,g\right)$ and $\left(S,\iota^{*}g\right)$
respectively, we have the pointwise inequality
\[
\kappa_{\infty}^{\iota^{*}g}\geq\kappa_{\infty|\partial S}^{g}.
\]
\end{prop}

\begin{proof}
Since $S$ is an interior p-submanifold, $\iota:S\to M$ is a simple
$b$-map and an injective immersion. Therefore, the bundle $^{0}TS$
can be identified via $^{0}d\iota$ with a subbundle of ${^{0}TM_{|S}}$.
The pull-back of $g$ is precisely the restriction of $g$ to $^{0}TS$:
it is a bundle metric on $^{0}TS$, hence a conformally compact metric
on $S$ as claimed. Now, let $X$ be a boundary defining function
for $M$, and let $x=\iota^{*}X$. Call $h=\iota^{*}g$. Let $\overline{h}=x^{2}h$
and $\overline{g}=X^{2}g$. Then $\iota$ is an isometric embedding
$\left(S,\overline{h}\right)\to\left(M,\overline{g}\right)$, and
moreover by definition we have
\begin{align*}
\kappa_{\infty}^{g} & =-\left|\nabla^{\overline{g}}X\right|_{\overline{g}|\partial M}^{2}\\
\kappa_{\infty}^{h} & =-\left|\nabla^{\overline{h}}x\right|_{\overline{h}|\partial S}^{2}.
\end{align*}
Thus, it suffices to prove that $\left|\nabla^{\overline{h}}x\right|_{\overline{h}}^{2}\leq\iota^{*}\left|\nabla^{\overline{g}}X\right|_{\overline{g}}^{2}$
along the boundary. This is immediate: in general, if $f\in C^{\infty}\left(M\right)$,
the gradient of $f_{|S}$ is precisely the orthogonal projection to
$TS$ of the gradient of $f$.
\end{proof}
Corollary \ref{cor:p-submanifold-iff-image-immersion-bmap} and Proposition
\ref{prop:p-submanifolds-of-CC-are-CC} justify the following definition:
\begin{defn}
Let $M,N$ be compact manifolds with boundary. An \emph{interior p-embedding}
$u:M\to N$ is a simple $b$-map which is also an injective immersion.
If $g$ and $h$ are conformally compact metrics on $M$ and $N$
respectively, $u$ is said to be an \emph{isometric interior p-embedding}
if $u^{*}h=g$.
\end{defn}

\section{Isometric p-Embeddings into Hyperbolic Spaces}

The $C^{\infty}$ Nash Isometric Embedding Theorem states that every
closed Riemannian manifold admits an isometric embedding into a sufficiently
high-dimensional Euclidean space. The following definition is therefore
well-posed:
\begin{defn}
Let $m\in\mathbb{N}$. We denote by $\mathrm{N}_{m}$ the smallest
positive integer $N$ such that every closed $m$-dimensional Riemannian
manifold $\left(M,g\right)$ admits an isometric embedding into $\mathbb{R}^{N}$.
\end{defn}

We are ready to state the main theorem of this paper. Given a non-zero
real number $\lambda$, denote by $\mathrm{H}^{N+1}\left(-\lambda^{2}\right)$
the hyperbolic space of constant sectional curvature $-\lambda^{2}$.
\begin{thm}
\label{thm:main-thm}Let $\left(M^{m},g\right)$ be a $m$-dimensional
conformally compact manifold. Suppose that the sectional curvature
at infinity of $g$ satisfies the pointwise inequality $\kappa_{\infty}^{g}>-\lambda^{2}$.
Then there exists an isometric interior p-embedding $u:\left(M,g\right)\to\mathrm{H}^{\mathrm{N}_{m}+1}\left(-\lambda^{2}\right)$.
\end{thm}

The crucial step in the proof is the existence of a boundary defining
function for $\left(M,g\right)$ with special properties.
\begin{lem}
\label{lem:very-special-bdf}Let $\left(M,g\right)$ be a conformally
compact manifold, and assume that $\kappa_{\infty}^{g}>-1$. Let $h_{0}$
be a metric in the conformal infinity of $\left(M,g\right)$. Then
there exists a smooth boundary defining function $x$ on $M$ with
the following properties:
\begin{enumerate}
\item the conformally compactified metric $x^{2}g$ on $M$ induces the
metric $h_{0}$ on $\partial M$;
\item the symmetric $0$-$2$-tensor on $M$
\[
g-\frac{dx^{2}}{x^{2}}
\]
is positive-definite, i.e. it is a smooth conformally compact metric
on $M$.
\end{enumerate}
\end{lem}

\begin{proof}
Choose an auxiliary boundary defining function $r$ for $M$ such
that $\overline{g}:=r^{2}g$ induces the metric $h_{0}$ on $\partial M$,
and define $K^{2}=\left|dr\right|_{\overline{g}}^{2}$. Call $\overline{\nabla}r$
the $\overline{g}$-gradient of $r$, and define $\overline{V}=K^{-2}\overline{\nabla}r$.
Since $r$ is a boundary defining function, $K^{2}$ is nowhere zero
near $\partial M$, and therefore $\overline{V}$ is well-defined
near $\partial M$ as well; moreover, $\overline{V}$ is transverse
to $\partial M$ and inward-pointing. Therefore, we can use the flow
of $\overline{V}$ starting at $\partial M$ to define a collar
\[
[0,\varepsilon)\times\partial M\hookrightarrow M
\]
of $\partial M$ in $M$. Note that, in this collar, we have 
\begin{align*}
\overline{V}r & =\left|dr\right|_{\overline{g}}^{-2}\left(\overline{\nabla}r\right)r\\
 & =\left|dr\right|_{\overline{g}}^{-2}\left|\overline{\nabla}r\right|_{\overline{g}}^{2}\\
 & \equiv1.
\end{align*}
Therefore, we can identify $r$ with the half-line coordinate in the
collar $[0,\varepsilon)\times\partial M$, and we can identify $\overline{V}$
with $\partial_{r}$. Since $\overline{V}$ is proportional to $\overline{\nabla}r$,
the level sets of $r$ are orthogonal to $\overline{V}$; therefore,
the metric $\overline{g}$ can be written in the collar as
\[
\overline{g}=\left|dr\right|_{\overline{g}}^{-2}dr^{2}+h\left(r\right),
\]
where $h:[0,\varepsilon)\to C^{\infty}\left(\partial M;S^{2}\left(T^{*}\partial M\right)\right)$
is a smooth path of metrics starting at $h_{0}$. Dividing by $r^{2}$,
we obtain the following form for the conformally compact metric $g$:
\[
g=\frac{1}{K^{2}}\frac{dr^{2}}{r^{2}}+\frac{h\left(r\right)}{r^{2}}.
\]
This is similar to the normal form for asymptotically hyperbolic metric
used, for example, in \cite{GrahamLee, Graham_Volume}, where $K^{2}\equiv1$.
In our case, $K^{2}$ might be non-constant. However, since
\[
K_{|\partial M}^{2}=\left|\frac{dr}{r}\right|_{g|\partial M}^{2}=-\kappa_{\infty}^{g},
\]
we have $K_{|\partial M}^{2}<1$. By choosing $\varepsilon$ small
enough, we can then ensure that $K^{2}<1$ everywhere in the collar.

We will now find a boundary defining function $x$ satisfying the
properties claimed in the statement of the lemma, which on the collar
takes the form $x=e^{\varphi\left(r\right)}r$ for some smooth function
$\varphi:[0,\varepsilon)\to\mathbb{R}$. On the collar, we can write
\[
dx=e^{\varphi}\left(1+r\partial_{r}\varphi\right)dr;
\]
this implies that $x$ is indeed a boundary defining function, because
$dx_{|\partial M}=\left(e^{\varphi}dr\right)_{|\partial M}$. Therefore,
we can write
\begin{equation}
g-\frac{dx^{2}}{x^{2}}=\left(\frac{1}{K^{2}}-\left(1+r\partial_{r}\varphi\right)^{2}\right)\frac{dr^{2}}{r^{2}}+\frac{h\left(r\right)}{r^{2}}\label{eq:adjusted-metric}
\end{equation}
Suppose that we can find a smooth function $\varphi:[0,\varepsilon)\to\mathbb{R}$
such that:
\begin{enumerate}
\item $\varphi\left(0\right)=0$;
\item $\left(1+r\partial_{r}\varphi\right)^{2}\leq1$;
\item $1+r\partial_{r}\varphi\equiv0$ for $r\geq\varepsilon/2$.
\end{enumerate}
The fact that $1+r\partial_{r}\varphi\equiv0$ for $r\geq\varepsilon/2$
implies that $dx\equiv0$ on the portion $[\varepsilon/2,\varepsilon)\times\partial M$
of the collar, and therefore, we can extend $x$ smoothly to the whole
of $M$ in such a way that $x$ is constant on the complement of the
sub-collar $[0,\varepsilon/2)\times\partial M$. On this complement,
we have
\[
g-\frac{dx^{2}}{x^{2}}\equiv g
\]
and therefore this $0$-$2$-tensor is certainly positive-definite
there. Now, the condition $\left(1+r\partial_{r}\varphi\right)^{2}\leq1$
together with the inequality $K^{2}<1$ implies that
\[
\frac{1}{K^{2}}-\left(1+r\partial_{r}\varphi\right)^{2}>0.
\]
This guarantees that the symmetric $0$-$2$-tensor (\ref{eq:adjusted-metric})
is positive-definite in the collar. Finally, the fact that $\varphi\left(0\right)=0$
guarantees that $x^{2}g$ and $r^{2}g$ induce the same metric $h_{0}$
on the boundary.

It is very easy to see that a function $\varphi$ that satisfies the
three points above exists. Choose any smooth function $Q:[0,\varepsilon)\to\left[0,1\right]$
such that $Q\left(0\right)=0$ and $Q\left(r\right)=1$ for $r\geq\varepsilon/2$.
Since $Q$ vanishes at the origin, Taylor's Theorem implies that the
function $Q\left(r\right)/r$ is still smooth as a map $[0,\varepsilon)\to\mathbb{R}$.
Now, define
\[
\varphi\left(r\right)=-\int_{0}^{r}\frac{Q\left(t\right)}{t}dt.
\]
Then $\varphi$ is smooth, it satisfies $\varphi\left(0\right)=0$,
and moreover we have
\[
\partial_{r}\varphi=-\frac{Q\left(r\right)}{r},
\]
so that
\[
1+r\partial_{r}\varphi=1-Q\left(r\right).
\]
This function is bounded between $0$ and $1$, and vanishes identically
for $r\geq\varepsilon/2$.
\end{proof}
We can finally complete the proof of the main result:


\begin{proof}[Proof of Theorem~\ref{thm:main-thm}]First of all, we
can assume without loss of generality that $\lambda=1$. To see this,
denote by $h_{-\lambda^{2}}$ the hyperbolic $0$-metric on the compact
unit ball $B^{\mathrm{N}_{m}+1}$. Then $h_{-\lambda^{2}}$ is related
to the standard hyperbolic metric $h_{-1}$ by the constant rescaling
\[
h_{-\lambda^{2}}=\lambda^{-2}h_{-1}.
\]
Therefore, an interior p-embedding $u:M\to B^{\mathrm{N}_{m}+1}$
is isometric as a map $\left(M,g\right)\to\text{H}^{\mathrm{N}_{m}+1}\left(-\lambda^{2}\right)$
if and only if it is isometric as a map $\left(M,\lambda^{2}g\right)\to\text{H}^{\mathrm{N}_{m}+1}\left(-1\right)$.
Now, choosing an auxiliary boundary defining function $r$, we see
that the sectional curvature at infinity rescales as
\begin{align*}
\kappa_{\infty}^{\lambda^{2}g} & =-\left|\frac{dr}{r}\right|_{\lambda^{2}g|\partial M}^{2}\\
 & =-\lambda^{-2}\left|\frac{dr}{r}\right|_{g|\partial M}^{2}\\
 & =\lambda^{-2}\kappa_{\infty}^{g}.
\end{align*}
Therefore, we have $\kappa_{\infty}^{g}>-\lambda^{2}$ if and only
if $\kappa_{\infty}^{\lambda^{2}g}>-1$.

Assume from now on that $\kappa_{\infty}^{g}>-1$. We can further
simplify the argument by working in the half-space model of hyperbolic
space. This is the manifold $[0,+\infty)\times\mathbb{R}^{\mathrm{N}_{m}}$,
with coordinates $\left(X,\boldsymbol{Y}\right)$ with $\boldsymbol{Y}=\left(Y^{1},...,Y^{\mathrm{N}_{m}}\right)$
and equipped with the hyperbolic $0$-metric
\[
h_{-1}=\frac{dX^{2}+d\boldsymbol{Y}^{2}}{X^{2}}.
\]
In this model, one of the points at infinity is the direction $X^{2}+\left|\boldsymbol{Y}\right|^{2}\to+\infty$.
However, we can ignore this point, since $M$ is compact and therefore
any interior $p$-embedding $u:M\to[0,+\infty)\times\mathbb{R}^{\mathrm{N}_{m}}$
has compact image.

The proof is then reduced to find an interior p-embedding $u:M\to[0,+\infty)\times\mathbb{R}^{\mathrm{N}_{m}}$
satisfying $u^{*}h_{-1}=g$. By Lemma \ref{lem:very-special-bdf},
we can choose a boundary defining function $x$ on $M$ such that
the symmetric $0$-$2$-tensor $g-dx^{2}/x^{2}$ is positive-definite,
i.e. the symmetric $2$-tensor $G:=x^{2}g-dx^{2}$ is a metric on
$M$. Now, denote by $\tilde{M}$ a closed $m$-dimensional manifold
containing $M$ as a submanifold; for example, we could take $\tilde{M}$
to be the smooth double of $M$. Let $\tilde{G}$ be a smooth metric
on $\tilde{M}$ whose restriction to $M$ is $G$. By the Nash Embedding
Theorem, there exists a smooth isometric embedding $\tilde{v}:\left(\tilde{M},\tilde{G}\right)\to\mathbb{R}^{\mathrm{N}_{m}}$.
Restricting $\tilde{v}$ to $M\subset\tilde{M}$, we get an isometric
embedding $v:\left(M,G\right)\to\mathbb{R}^{\mathrm{N}_{m}}$. The
Euclidean metric on $\mathbb{R}^{\mathrm{N}_{m}}$ with coordinates
$\boldsymbol{Y}=\left(Y^{1},...,Y^{\mathrm{N}_{m}}\right)$ is precisely
$d\boldsymbol{Y}^{2}$. Therefore, the fact that $v:\left(M,G\right)\to\mathbb{R}^{\mathrm{N}_{m}}$
is an isometric embedding means precisely that $dv^{2}=G$.

Define now the map
\begin{align*}
u:M & \to[0,+\infty)\times\mathbb{R}^{\mathrm{N}_{m}}\\
p & \mapsto\left(x\left(p\right),v\left(p\right)\right).
\end{align*}
We claim that this map is an isometric interior p-embedding $\left(M,g\right)\to\mathrm{H}^{\mathrm{N}_{m}+1}\left(-1\right)$.
First of all, $u$ is clearly a simple $b$-map: by construction,
the pull-back of the boundary defining function $X$ for $\mathrm{H}^{\mathrm{N}_{m}+1}\left(-1\right)$
is $x$, a boundary defining function for $M$ by construction. Now,
$u$ is an injective immersion, because $v$ is an injective immersion.
Therefore, $u$ is an interior p-embedding. It remains to check that
$u$ is indeed isometric. By construction, the conformally compact
metric on $M$ induced by $u$ is
\[
u^{*}h_{-1}=\frac{dx^{2}}{x^{2}}+\frac{dv^{2}}{x^{2}};
\]
substituting $dv^{2}=G=x^{2}g-dx^{2}$, we obtain the claim.\end{proof}

We conclude with a natural remark arising from the previous theorem.
Suppose that the sectional curvature at infinity $\kappa_{\infty}^{g}$
satisfies the pointwise inequality $\kappa_{\infty}^{g}\geq-\lambda^{2}$,
and there is at least a point $p\in\partial M$ where the equality
is achieved. In this case, the proof given in this paper breaks down.
To see this, by possibly rescaling the metric, we can assume without
loss of generality that $\kappa_{\infty}^{g}\geq-1$ and that $\kappa_{\infty}^{g}\left(p\right)=-1$.
As in the proof of Lemma \ref{lem:very-special-bdf}, choose a boundary
defining function $r$, set $\overline{g}=r^{2}g$, and use the flow
of the rescaled gradient $\overline{V}:=\left|dr\right|_{\overline{g}}^{-2}\overline{\nabla}r$
to write $g$ in the form
\[
g=\frac{1}{\left|dr\right|_{\overline{g}}^{2}}\frac{dr^{2}}{r^{2}}+\frac{h\left(r\right)}{r^{2}}
\]
in a small collar of $\partial M$ in $M$ induced by $\overline{V}$.
Lemma \ref{lem:very-special-bdf} implies that, if $\kappa_{\infty}^{g}>-1$,
then there exists a boundary defining function $x$ such that the
$0$-$2$-tensor $g-\left(dx/x\right)^{2}$ is positive-definite.
However, in the present situation, for \emph{every} boundary defining
function $x$ the $0$-$2$-tensor $g-dx^{2}/x^{2}$ degenerates at
$p$. Indeed, in the collar we can write
\[
g-\frac{dx^{2}}{x^{2}}=\frac{1}{\left|dr\right|_{\overline{g}}^{2}}\frac{dr^{2}}{r^{2}}-\frac{dx^{2}}{x^{2}}+\frac{h\left(r\right)}{r^{2}}.
\]
As explained in Remark \ref{rem:dx/x_canonical}, the $0$-$1$-forms
$dx/x$ and $dr/r$ coincide on $\partial M$; moreover, by definition,
$\kappa_{\infty}^{g}=-\left|dr\right|_{\overline{g}|\partial M}^{2}$,
and $\left|dr\right|_{\overline{g}}^{2}\left(p\right)=1$ by hypothesis.
It follows that $\left(g-dx^{2}/x^{2}\right)_{p}$ is not positive-definite;
therefore, Lemma \ref{lem:very-special-bdf} cannot be extended to
the borderline case where $\kappa_{\infty}^{g}=-1$ at some point
of $\partial M$. This, of course, does not imply that Theorem \ref{thm:main-thm}
fails when the equality in $\kappa_{\infty}^{g}\geq-\lambda^{2}$
is attained at some boundary point. We therefore close with the following
question:

\begin{question}Let $\left(M,g\right)$ be conformally compact, and
assume that $\kappa_{\infty}^{g}\geq-\lambda^{2}$ for some constant
$\lambda>0$. Can $\left(M,g\right)$ be isometrically p-embedded
into the hyperbolic space $\mathrm{H}^{\mathrm{N}_{m}+1}\left(-\lambda^{2}\right)$?
In particular, if $\left(M,g\right)$ is asymptotically hyperbolic,
can it be isometrically p-embedded into $\mathrm{H}^{\mathrm{N}_{m}+1}\left(-1\right)$?\end{question}

The impression of the author is that the answer to this question is
\emph{yes}; however, proving such a result (if true) may require much
more analytic work. Günther cleverly manages to reduce the proof of
the $C^{\infty}$ Nash Embedding Theorem to a standard fixed point
argument, by making appropriate use of elliptic pseudodifferential
theory on a closed manifold. A possible path towards the proof could
be to try and adapt Günther's proof to the present context, using
the elliptic theory of $0$-pseudodifferential operators of Mazzeo--Melrose
\cite{MazzeoMelroseResolvent, MazzeoPhD, MazzeoEdgeI} in place of
the standard elliptic theory on closed manifolds.

\bibliography{allpapers}{}
\bibliographystyle{alpha}
\newpage
\end{document}